\documentclass[11pt,reqno,twoside]{amsart}
\usepackage{amssymb,amsmath,amsthm,soul,color,paralist}
\usepackage{t1enc}
\usepackage[cp1250]{inputenc}
\usepackage{a4,indentfirst,latexsym}
\usepackage{graphics}
\usepackage{mathrsfs}
\usepackage{cite,enumitem,graphicx}
\usepackage[colorlinks=true,urlcolor=blue,
citecolor=red,linkcolor=blue,linktocpage,pdfpagelabels,
bookmarksnumbered,bookmarksopen]{hyperref}
\usepackage[english]{babel}
\usepackage[left=2.50cm,right=2.50cm,top=2.72cm,bottom=2.72cm]{geometry}
\usepackage[metapost]{mfpic}
\usepackage[hyperpageref]{backref}
\usepackage[colorinlistoftodos]{todonotes}
\usepackage[normalem]{ulem}

\makeatletter
\providecommand\@dotsep{5}
\def\listtodoname{List of Todos}
\def\listoftodos{\@starttoc{tdo}\listtodoname}
\makeatother

\numberwithin{equation}{section}

\newcommand{\R}{\mathbb{R}}

\newcommand{\q}{q^{*}_{s}}
\newcommand{\2}{p^{*}_{s}}

\DeclareMathOperator{\dive}{div}
\DeclareMathOperator{\supp}{supp}
\DeclareMathOperator{\e}{\varepsilon}

\newtheorem{lem}{Lemma}[section]
\newtheorem{thm}{Theorem}[section]
\newtheorem{defn}{Definition}[section]

\title[Fractional $p\&q$ Laplacian problems in $\R^{N}$ with critical growth]{Fractional $p\&q$ Laplacian problems in $\mathbb{R}^{N}$ with critical growth}

\author[V. Ambrosio]{Vincenzo Ambrosio}
\address{Vincenzo Ambrosio\hfill\break\indent
Dipartimento di Ingegneria Industriale e Scienze Matematiche \hfill\break\indent
Universit\`a Politecnica delle Marche\hfill\break\indent
Via Brecce Bianche, 12\hfill\break\indent
60131 Ancona (Italy)}
\email{ambrosio@dipmat.univpm.it}


\begin{document}

\maketitle

\begin{abstract}
We deal with the following nonlinear problem involving fractional $p\&q$ Laplacians:
\begin{equation*}
(-\Delta)^{s}_{p}u+(-\Delta)^{s}_{q}u+|u|^{p-2}u+|u|^{q-2}u=\lambda h(x) f(u)+|u|^{q^{*}_{s}-2}u \mbox{ in } \mathbb{R}^{N},
\end{equation*}
where $s\in (0,1)$, $1<p<q<\frac{N}{s}$, $\q=\frac{Nq}{N-sq}$, $\lambda>0$ is a parameter, $h$ is a nontrivial bounded perturbation and $f$ is a superlinear continuous function with subcritical growth. Using suitable variational arguments and concentration-compactness lemma, we prove the existence of a nontrivial non-negative solution for $\lambda$ sufficiently large.
\end{abstract}

\maketitle

\section{introduction}
In this paper we consider  the following fractional nonlinear problem
\begin{equation}\label{P}
(-\Delta)^{s}_{p}u+(-\Delta)^{s}_{q}u+|u|^{p-2}u+|u|^{q-2}u=\lambda h(x) f(u)+|u|^{q^{*}_{s}-2}u \mbox{ in } \R^{N},
\end{equation}
where $s\in (0,1)$, $1<p<q< \frac{N}{s}$, $\q=\frac{Nq}{N-sq}$, $\lambda>0$ is a parameter, $f: \R\rightarrow \R$ is a continuous function and $h:\R^{N}\rightarrow \R$ is a nontrivial bounded function. \\
Here the nonlocal operator $(-\Delta)^{s}_{t}$, with $t\in \{p, q\}$, is the fractional $t$-Laplacian operator which is defined, up to a normalizing constant, by 
$$
(-\Delta)^{s}_{t}u(x)= P.V. \int_{\R^{N}} \frac{|u(x)-u(y)|^{t-2}(u(x)-u(y))}{|x-y|^{N+st}} dy
$$ 
for any $u: \R^{N}\rightarrow \R$ sufficiently smooth; see \cite{Caff} for more motivations on this operator.\\
When $s=1$, equation \eqref{P} becomes a $p\&q$ elliptic problem of the form
\begin{equation}\label{pq}
-\Delta_{p}u-\Delta_{q} u+|u|^{p-2}u+|u|^{q-2}u= f(x,u) \mbox{ in } \R^{N}.
\end{equation} 
As explained in \cite{CIL}, the study of equation \eqref{pq} is motivated by the more general reaction-diffusion system:
$$
u_{t}=\dive(D(u)\nabla u)+c(x,u) \mbox{ and } D(u)=|\nabla u|^{p-2}+|\nabla u|^{q-2},
$$ 
which finds applications in biophysics, plasma physics and chemical reaction design. In these contexts, $u$  represents a concentration, $\dive(D(u) \nabla u)$ is the diffusion with diffusion coefficient $D(u)$, and the reaction term $c(x, u)$ relates to source and loss processes. We recall that classical $p\&q$ Laplacian problems in bounded domains and in the whole of $\R^{N}$ have been  widely investigated by many authors; see for instance \cite{AF, BF, CIL, CCF, Fig1, Fig2, LG, LL, MP} and the references therein. \\
On the other hand, in the last years a great attention has been devoted to the study of the fractional $p$-Laplacian operator.
For instance, fractional $p$-eigenvalue problems have been considered in \cite{FrP, LLq}. Some interesting regularity results for weak solutions can be found in \cite{DKP2, IMS}. Several existence and multiplicity results for problems set in bounded domains or in the whole of $\R^{N}$ have been established in \cite{A1, A2, A5, AI, BMB, DelpezzoQuaas, PXZ, Torres}.
For more details on fractional operators and the corresponding nonlocal problems, we refer the interested reader to \cite{DPV, MBRS}.\\
Motivated by the above papers, in this work we are interested in the existence of nontrivial solutions for a fractional $p\&q$ Laplacian problem involving the critical exponent. 
To our knowledge, only one result for fractional $p\&q$ problems is present in the literature \cite{CB}, in which the authors studied existence, nonexistence and multiplicity for a nonlocal subcritical problem. 
The aim of this paper is to give a further result for this interesting class of fractional problems in the case of critical growth. 
%We point out that classical $p\&q$ Laplacian problems have been  widely investigated by many authors; see for instance \cite{BF, CIL, CCF, Fig1, Fig2, LG, LL, MP} and the references therein.
%The aim of this paper is to give a first result regarding this class of interesting fractional problems. 

Before stating our main result, we introduce the assumptions on the nonlinearity $f$. We suppose that $f: \R\rightarrow \R$ is a continuous function such that $f(t)=0$ for $t\leq 0$ and 
\begin{compactenum}[$(f_1)$]
\item $\lim_{|t|\rightarrow 0} \frac{|f(t)|}{|t|^{p-1}}=0$,
\item there exists $r\in (q, q^{*}_{s})$ such that $\lim_{|t|\rightarrow \infty} \frac{|f(t)|}{|t|^{r-1}}=0$,
\item there exists $\theta\in (q, q^{*}_{s})$ such that $0<\theta F(t)\leq f(t)t$ for all $t>0$, where $F(t)=\int_{0}^{t} f(\tau)d\tau$,
\end{compactenum}
and we assume that  $h:\R^{N}\rightarrow \R$ fulfills the following condition:
\begin{compactenum}[$(h)$]
\item $h\geq 0$, $h\not\equiv 0$ in $\R^{N}$ and $h\in L^{\infty}(\R^{N})\cap L^{\frac{\q}{\q-r}}(\R^{N})$.
\end{compactenum}
In order to find weak solutions to \eqref{P}, we look for critical points of the Euler-Lagrange functional $J: X\rightarrow \R$ defined as
$$
J(u)=\frac{1}{p}\|u\|^{p}_{s,p}+\frac{1}{q}\|u\|^{q}_{s,q}-\lambda\int_{\R^{N}} h(x) F(u)dx-\frac{1}{q^{*}_{s}} |u|_{q^{*}_{s}}^{q^{*}_{s}},
$$
where 
$$
\|u\|_{s,r}=([u]^{r}_{s, r}+|u|_{r}^{r})^{\frac{1}{r}},
$$
$$
[u]^{r}_{s, r}=\iint_{\R^{2N}} \frac{|u(x)-u(y)|^{r}}{|x-y|^{N+sr}} dxdy \quad \mbox{ and } \quad |u|_{r}^{r}=\int_{\R^{N}} |u|^{r} dx.
$$
Here we denote by $X=W^{s, p}(\R^{N})\cap W^{s,q}(\R^{N})$ the space endowed with the norm
$$
\|u\|:=\|u\|_{s,p}+\|u\|_{s,q}.
$$
Then we give the following definition:
\begin{defn}
We say that $u\in X$ is a weak solution to \eqref{P} if for any $v\in X$ we have
\begin{align*}
&\iint_{\R^{2N}} \frac{|u(x)-u(y)|^{p-2}(u(x)-u(y))}{|x-y|^{N+sp}}(v(x)-v(y))dxdy+\int_{\R^{N}} |u|^{p-2}uvdx \\
&+\iint_{\R^{2N}} \frac{|u(x)-u(y)|^{q-2}(u(x)-u(y))}{|x-y|^{N+sq}}(v(x)-v(y))dxdy+\int_{\R^{N}} |u|^{q-2}uvdx \\
&=\lambda \int_{\R^{N}} h(x) f(u)vdx+\int_{\R^{N}} |u|^{\q-2}u v dx.
\end{align*}
\end{defn}

\noindent
The main result of this paper can be stated as follows:
\begin{thm}\label{thm1}
Assume that $(f_1)$-$(f_3)$ and $(h)$ hold. Then there exists $\lambda^{*}>0$ such that problem \eqref{P} admits a nontrivial non-negative solution for all $\lambda\geq \lambda^{*}$.
\end{thm}
The proof of Theorem \ref{thm1} is obtained via variational methods inspired by \cite{Fig1}.
Anyway, the presence of two fractional Laplacian operators, the perturbation $h$ and the lack of compactness due to the critical exponent make our analysis rather delicate and intriguing; see the proof of Lemma \ref{lem3}. 
More precisely, we prove a variant of the concentration-compactness lemma \cite{Lions}, which takes care of the possible loss of mass at infinity in the spirit of \cite{Ch} (see Lemma \ref{CCL}), and we show that weak limits of Palais-Smale sequences of $J$ are weak solutions to \eqref{P} in a different way with respect to the one given in \cite{Fig1} which is based on some "local" arguments developed in \cite{LG}. We are also able to cover the case $1<p<q$ which has not been considered in \cite{Fig1}.\\
The paper is organized as follows: in Section $2$ we recall some useful lemmas which will be used along the paper. In particular, we give a variant of the concentration-compactness lemma \cite{Lions} for the fractional $p$-Laplacian. In Section $3$ we show that \eqref{P} admits a nontrivial solution for $\lambda$ big enough by applying the mountain pass theorem \cite{AR} and a suitable version of the Lions' compactness result \cite{Lionsll}.

\section{preliminaries}
In this section, we collect some useful results about fractional Sobolev spaces. For more details we refer the interested reader to \cite{DPV, MBRS}.\\
Let us define $D^{s, p}(\R^{N})$ as the completion of $C^{\infty}_{c}(\R^{N})$ with respect to the norm
$$
[u]^{p}_{s, p}=\iint_{\R^{2N}} \frac{|u(x)-u(y)|^{p}}{|x-y|^{N+sp}} dxdy.
$$
We denote by $W^{s,p}(\R^{N})$ the set of functions $u:\R^{N}\rightarrow \R$ belonging to $L^{p}(\R^{N})$ such that $[u]_{s,p}<\infty$. 
Let us recall the following fundamental embeddings:
\begin{thm}\label{Sembedding}\cite{DPV}
Let $N>s p$. Then there exists a constant $S_{*}=S_{*}(N, s, p)>0$ such that 
$$
S_{*} |u|_{p^{*}_{s}}^{p}\leq [u]_{s,p}^{p} \quad \forall u\in D^{s, p}(\R^{N}).
$$
Moreover, the space $W^{s, p}(\R^{N})$ is continuously embedded in $L^{t}(\R^{N})$ for any $t\in [p, p^{*}_{s}]$, and compactly embedded in $L^{t}(\R^{N})$ for any $t\in [1, p^{*}_{s})$. 
\end{thm}

\noindent
%A simple adaption of the arguments in \cite{Lions} allows us to deduce the following useful result:
%\begin{lem}\label{Llem}
%Let $R>0$ and $r\in [p, p^{*}_{s})$. For any bounded sequence $(u_{n})\subset W^{s,p}(\R^{N})$, if 
%$$
%\sup_{y\in \R^{N}} \int_{B_{R}(y)} |u_{n}|^{r} dx\rightarrow 0 \quad \mbox{ as } n\rightarrow \infty,
%$$
%then $u_{n}\rightarrow 0$ in $L^{q}(\R^{N})$ for any $q\in (r, p^{*}_{s})$.
%\end{lem}
%\begin{proof}
%Applying H\"older inequality we can see that
%\begin{align*}
%\int_{B_{r}(y)} |u_{n}|^{q}dx&\leq \left(\int_{B_{r}(y)} |u_{n}|^{r}dx\right)^{\frac{1-\lambda}{p}} \left(\int_{B_{r}(y)} |u_{n}|^{p^{*}_{s}}dx\right)^{\frac{\lambda}{p^{*}_{s}}},
%\end{align*}
%where $\frac{1-\lambda}{r}+\frac{\lambda}{p^{*}_{s}}=\frac{1}{q}$. Then, covering $\R^N$ by balls with radius $r$ in such a way that each point of $\R^N$ is contained in at most $N+1$ balls and using Theorem \ref{Sembedding} we have
%$$
%\int_{\R^{N}} |u_{n}|^{q}dx\leq C(N+1) \sup_{y\in \R^{N}} \left(\int_{B_{r}(y)} |u_{n}|^{r} dx\right)^{\frac{1-\lambda}{p}},
%$$
%which implies the conclusion.
%\end{proof}
Now we prove the following technical lemmas (see also \cite{AI, ZZR} for related results).
\begin{lem}\label{etaR}
Let $(u_{n})\subset D^{s,p}(\R^{N})$ be a bounded sequence and $\eta\in C^{\infty}(\R^{N})$ be such that $\eta=0$ in $B_{1}$ and $\eta=1$ in $\R^{N}\setminus B_{2}$, and we set $\eta_{R}(x)=\eta(x/R)$. Then
$$
\lim_{R\rightarrow \infty} \limsup_{n\rightarrow \infty} \iint_{\R^{2N}} \frac{|\eta_{R}(x)-\eta_{R}(y)|^{p}}{|x-y|^{N+sp}}|u_{n}(x)|^{p} dx dy=0.
$$
\end{lem}
\begin{proof}
Firstly, we note that  $\R^{2N}$ can be written as follows:
$$
\R^{2N}=((\R^{N}\setminus B_{2R})\times (\R^{N}\setminus B_{2R})) \cup ((\R^{N}\setminus B_{2R})\times B_{2R})\cup (B_{2R}\times \R^{N})=: X^{1}_{R}\cup X^{2}_{R} \cup X^{3}_{R}.
$$
Therefore,
\begin{align}\label{Pa1}
&\iint_{\R^{2N}}\frac{|\eta_{R}(x)-\eta_{R}(y)|^{p}}{|x-y|^{N+sp}} |u_{n}(x)|^{p} dx dy =\iint_{X^{1}_{R}}\frac{|\eta_{R}(x)-\eta_{R}(y)|^{p}}{|x-y|^{N+sp}} |u_{n}(x)|^{p} dx dy \nonumber \\
&+\iint_{X^{2}_{R}}\frac{|\eta_{R}(x)-\eta_{R}(y)|^{p}}{|x-y|^{N+sp}} |u_{n}(x)|^{p} dx dy+
\iint_{X^{3}_{R}}\frac{|\eta_{R}(x)-\eta_{R}(y)|^{p}}{|x-y|^{N+sp}} |u_{n}(x)|^{p} dx dy.
\end{align}
Now, we estimate each integral in \eqref{Pa1}.
Since $\eta_{R}=1$ in $\R^{N}\setminus B_{2R}$, we have
\begin{align}\label{Pa2}
\iint_{X^{1}_{R}}\frac{|u_{n}(x)|^{p}|\eta_{R}(x)-\eta_{R}(y)|^{p}}{|x-y|^{N+sp}} dx dy=0.
\end{align}
Take $k>4$. Clearly,
\begin{equation*}
X^{2}_{R}\subset (\R^{N} \setminus B_{2R})\times B_{2R} \subset ((\R^{N}\setminus B_{kR})\times B_{2R})\cup ((B_{kR}\setminus B_{2R})\times B_{2R}) 
\end{equation*}
Let us note that if $(x, y) \in (\R^{N}\setminus B_{kR})\times B_{2R}$, then
\begin{equation*}
|x-y|\geq |x|-|y|\geq |x|-2R>\frac{|x|}{2}. 
\end{equation*}
Therefore, using the above fact, $0\leq \eta_{R}\leq 1$ and applying the H\"older inequality we obtain
\begin{align*}
\int_{\R^{N}\setminus B_{kR}} \int_{B_{2R}} \frac{|u_{n}(x)|^{p}|\eta_{R}(x)-\eta_{R}(y)|^{p}}{|x-y|^{N+sp}} dx dy&\leq C \int_{\R^{N}\setminus B_{kR}} \int_{B_{2R}} \frac{|u_{n}(x)|^{p}}{|x|^{N+sp}}\, dxdy \\
&\leq CR^{N} \int_{\R^{N}\setminus B_{kR}} \frac{|u_{n}(x)|^{p}}{|x|^{N+sp}}\, dx \\
&\leq CR^{N} \left( \int_{\R^{N}\setminus B_{kR}} |u_{n}(x)|^{p^{*}_{s}} dx \right)^{\frac{p}{p^{*}_{s}}} \left(\int_{\R^{N}\setminus B_{kR}}\frac{1}{|x|^{\frac{N^{2}}{sp} +N}}\, dx \right)^{\frac{sp}{N}}\\
&\leq \frac{C}{k^{N}} \left( \int_{\R^{N}\setminus B_{kR}} |u_{n}(x)|^{p^{*}_{s}} dx \right)^{\frac{p}{p^{*}_{s}}}\leq \frac{C}{k^{N}}.
\end{align*}
Note that by $|\nabla \eta_{R}|\leq \frac{C}{R}$ and using the H\"older inequality we also have
\begin{align*}
\int_{B_{kR}\setminus B_{2R}} \int_{B_{2R}} \frac{|u_{n}(x)|^{p}|\eta_{R}(x)-\eta_{R}(y)|^{p}}{|x-y|^{N+sp}} dx dy&\leq \frac{C}{R^{p}} \int_{B_{kR}\setminus B_{2R}} \int_{B_{2R}} \frac{|u_{n}(x)|^{p}}{|x-y|^{N+p(s-1)}}\, dxdy\\
&\leq \frac{C}{R^{p}} (kR)^{p(1-s)} \int_{B_{kR}\setminus B_{2R}} |u_{n}(x)|^{p} dx\\
&\leq \frac{C k^{p(1-s)}}{R^{sp}} \int_{B_{kR}\setminus B_{2R}} |u_{n}(x)|^{p} dx.
\end{align*}
Gathering the above estimates we get 
\begin{align}\label{Pa3}
&\iint_{X^{2}_{R}}\frac{|u_{n}(x)|^{p}|\eta_{R}(x)-\eta_{R}(y)|^{p}}{|x-y|^{N+sp}} dx dy \nonumber \\
&=\int_{\R^{N}\setminus B_{kR}} \int_{B_{2R}} \frac{|u_{n}(x)|^{p}|\eta_{R}(x)-\eta_{R}(y)|^{p}}{|x-y|^{N+sp}} dx dy + \int_{B_{kR}\setminus B_{2R}} \int_{B_{2R}} \frac{|u_{n}(x)|^{p}|\eta_{R}(x)-\eta_{R}(y)|^{p}}{|x-y|^{N+sp}} dx dy \nonumber \\
%&\leq C \int_{\R^{N}\setminus B_{kR}} \int_{B_{2R}} \frac{|u_{n}(x)|^{p}}{|x|^{N+sp}}\, dxdy+ \frac{C}{R^{p}} \int_{B_{kR}\setminus B_{2R}} \int_{B_{2R}} \frac{|u_{n}(x)|^{p}}{|x-y|^{N+p(s-1)}}\, dxdy \nonumber \\
%&\leq CR^{N} \int_{\R^{N}\setminus B_{kR}} \frac{|u_{n}(x)|^{2}}{|x|^{N+sp}}\, dx + \frac{C}{R^{2}} (kR)^{p(1-s)} \int_{B_{kR}\setminus B_{2R}} |u_{n}(x)|^{p} dx \nonumber \\
%&\leq CR^{N} \left( \int_{\R^{N}\setminus B_{kR}} |u_{n}(x)|^{p^{*}_{s}} dx \right)^{\frac{p}{p^{*}_{s}}} \left(\int_{\R^{N}\setminus B_{kR}}\frac{1}{|x|^{\frac{N^{2}}{sp} +N}}\, dx \right)^{\frac{sp}{N}} + \frac{C k^{p(1-s)}}{R^{2s}} \int_{B_{kR}\setminus B_{2R}} |u_{n}(x)|^{p} dx \nonumber \\
%&\leq \frac{C}{k^{N}} \left( \int_{\R^{N}\setminus B_{kR}} |u_{n}(x)|^{p^{*}_{s}} dx \right)^{\frac{p}{p^{*}_{s}}} + \frac{C k^{p(1-s)}}{R^{sp}} \int_{B_{kR}\setminus B_{2R}} |u_{n}(x)|^{p} dx \nonumber \\
&\leq \frac{C}{k^{N}}+ \frac{C k^{p(1-s)}}{R^{sp}} \int_{B_{kR}\setminus B_{2R}} |u_{n}(x)|^{p} dx.
\end{align}
On the other hand, 
\begin{align}\label{Ter1}
&\iint_{X^{3}_{R}} \frac{|u_{n}(x)|^{p} |\eta_{R}(x)- \eta_{R}(y)|^{p}}{|x-y|^{N+sp}}\, dxdy \nonumber\\
&\leq \int_{B_{2R}\setminus B_{R/k}} \int_{\R^{N}} \frac{|u_{n}(x)|^{p} |\eta_{R}(x)- \eta_{R}(y)|^{p}}{|x-y|^{N+sp}}\, dxdy + \int_{B_{R/k}} \int_{\R^{N}} \frac{|u_{n}(x)|^{p} |\eta_{R}(x)- \eta_{R}(y)|^{p}}{|x-y|^{N+sp}}\, dxdy. 
\end{align} 
Next, we estimate the integrals on the right hand side in \eqref{Ter1}. In view of
\begin{align*}
\int_{B_{2R}\setminus B_{R/k}} \int_{\R^{N} \cap \{y: |x-y|<R\}} \frac{|u_{n}(x)|^{p} |\eta_{R}(x)- \eta_{R}(y)|^{p}}{|x-y|^{N+sp}}\, dxdy \leq \frac{C}{R^{sp}} \int_{B_{2R}\setminus B_{R/k}} |u_{n}(x)|^{p} dx,
\end{align*}
and 
\begin{align*}
\int_{B_{2R}\setminus B_{R/k}} \int_{\R^{N} \cap \{y: |x-y|\geq R\}} \frac{|u_{n}(x)|^{p} |\eta_{R}(x)- \eta_{R}(y)|^{p}}{|x-y|^{N+sp}}\, dxdy \leq \frac{C}{R^{sp}} \int_{B_{2R}\setminus B_{R/k}} |u_{n}(x)|^{p} dx
\end{align*}
we can see that
\begin{align}\label{Ter2} 
\int_{B_{2R}\setminus B_{R/k}} \int_{\R^{N}} \frac{|u_{n}(x)|^{p} |\eta_{R}(x)- \eta_{R}(y)|^{p}}{|x-y|^{N+sp}}\, dxdy \leq \frac{C}{R^{sp}} \int_{B_{2R}\setminus B_{R/k}} |u_{n}(x)|^{p} dx. 
\end{align}
By the definition of $\eta_{R}$ and $0\leq \eta_{R}\leq 1$ we get 
\begin{align}\label{Ter3}
\int_{B_{R/k}} \int_{\R^{N}} \frac{|u_{n}(x)|^{p} |\eta_{R}(x)- \eta_{R}(y)|^{p}}{|x-y|^{N+sp}}\, dxdy &= \int_{B_{R/k}} \int_{\R^{N}\setminus B_{R}} \frac{|u_{n}(x)|^{p} |\eta_{R}(x)- \eta_{R}(y)|^{p}}{|x-y|^{N+sp}}\, dxdy\nonumber \\
&\leq C \int_{B_{R/k}} \int_{\R^{N}\setminus B_{R}} \frac{|u_{n}(x)|^{p}}{|x-y|^{N+sp}}\, dxdy\nonumber \\
&\leq C \int_{B_{R/k}} |u_{n}|^{p} dx \int_{(1-\frac{1}{k})R}^{\infty} \frac{1}{r^{1+sp}} dr\nonumber \\
&=\frac{C}{[(1-\frac{1}{k})R]^{sp}} \int_{B_{R/k}} |u_{n}|^{p} dx
\end{align}
where we used the fact that if $(x, y) \in B_{R/k}\times (\R^{N} \setminus B_{R})$ then $|x-y|>(1-\frac{1}{k})R$. \\
Thus, \eqref{Ter1}, \eqref{Ter2} and \eqref{Ter3} yield 
\begin{align}\label{Pa4}
\iint_{X^{3}_{R}} &\frac{|u_{n}(x)|^{p} |\eta_{R}(x)- \eta_{R}(y)|^{p}}{|x-y|^{N+sp}}\, dxdy \nonumber \\
&\leq \frac{C}{R^{sp}} \int_{B_{2R}\setminus B_{R/k}} |u_{n}(x)|^{p} dx + \frac{C}{[(1-\frac{1}{k})R]^{sp}} \int_{B_{R/k}} |u_{n}(x)|^{p} dx. 
\end{align}
In the light of \eqref{Pa1}, \eqref{Pa2}, \eqref{Pa3} and \eqref{Pa4} we can infer  
\begin{align}\label{Pa5}
\iint_{\R^{2N}} &\frac{|u_{n}(x)|^{p} |\eta_{R}(x)- \eta_{R}(y)|^{p}}{|x-y|^{N+sp}}\, dxdy \nonumber \\
&\leq \frac{C}{k^{N}} + \frac{Ck^{p(1-s)}}{R^{sp}} \int_{B_{kR}\setminus B_{2R}} |u_{n}(x)|^{p} dx + \frac{C}{R^{sp}} \int_{B_{2R}\setminus B_{R/k}} |u_{n}(x)|^{p} dx \nonumber \\
&+ \frac{C}{[(1-\frac{1}{k})R]^{sp}}\int_{B_{R/k}} |u_{n}(x)|^{p} dx. 
\end{align}
Since $(u_{n})$ is bounded in $D^{s,p}(\R^{N})$, we may suppose that $u_{n}\rightarrow u$ in $L^{p}_{loc}(\R^{N})$ for some $u\in D^{s,p}(\R^{N})$, thanks to Theorem \ref{Sembedding}. Then, taking the limit as $n\rightarrow \infty$ in \eqref{Pa5}, we obtain
\begin{align*}
&\limsup_{n\rightarrow \infty} \iint_{\R^{2N}} \frac{|u_{n}(x)|^{p} |\eta_{R}(x)- \eta_{R}(y)|^{p}}{|x-y|^{N+sp}}\, dxdy\\
&\leq \frac{C}{k^{N}} + \frac{Ck^{p(1-s)}}{R^{sp}} \int_{B_{kR}\setminus B_{2R}} |u(x)|^{p} dx + \frac{C}{R^{sp}} \int_{B_{2R}\setminus B_{R/k}} |u(x)|^{p} dx + \frac{C}{[(1-\frac{1}{k})R]^{sp}}\int_{B_{R/k}} |u(x)|^{p} dx \\
&\leq \frac{C}{k^{N}} + Ck^{p} \left( \int_{B_{kR}\setminus B_{2R}} |u(x)|^{p^{*}_{s}} dx\right)^{\frac{p}{p^{*}_{s}}} + C\left(\int_{B_{2R}\setminus B_{R/k}} |u(x)|^{p^{*}_{s}} dx\right)^{\frac{p}{p^{*}_{s}}} + C\left( \frac{1}{k-1}\right)^{sp} \left(\int_{B_{R/k}} |u(x)|^{p^{*}_{s}} dx\right)^{\frac{p}{p^{*}_{s}}}, 
\end{align*}
where in the last passage we used the H\"older inequality. \\
Since $u\in L^{p^{*}_{s}}(\R^{N})$ and $k>4$, it holds
\begin{align*}
\lim_{R\rightarrow \infty} \int_{B_{kR}\setminus B_{2R}} |u(x)|^{p^{*}_{s}} dx = \lim_{R\rightarrow \infty}  \int_{B_{2R}\setminus B_{R/k}} |u(x)|^{p^{*}_{s}} dx = 0, 
\end{align*}
from which we deduce that
\begin{align*}
&\lim_{R\rightarrow \infty} \limsup_{n\rightarrow \infty} \iint_{\R^{2N}} \frac{|u_{n}(x)|^{p} |\eta_{R}(x)- \eta_{R}(y)|^{p}}{|x-y|^{N+sp}}\, dxdy\\
&\leq \lim_{k\rightarrow \infty} \lim_{R\rightarrow \infty} \Bigl[\, \frac{C}{k^{N}} + Ck^{p} \left( \int_{B_{kR}\setminus B_{2R}} |u(x)|^{p^{*}_{s}} dx\right)^{\frac{p}{p^{*}_{s}}} + C\left(\int_{B_{2R}\setminus B_{R/k}} |u(x)|^{p^{*}_{s}} dx\right)^{\frac{p}{p^{*}_{s}}} \\
&\qquad+ C\left(\frac{1}{k-1}\right)^{sp} \left(\int_{B_{R/k}} |u(x)|^{p^{*}_{s}} dx\right)^{\frac{p}{p^{*}_{s}}}\, \Bigr]\\
&\leq \lim_{k\rightarrow \infty} \frac{C}{k^{N}} + C\left(\frac{1}{k-1}\right)^{sp} \left(\int_{\R^{N}} |u(x)|^{p^{*}_{s}} dx \right)^{\frac{p}{p^{*}_{s}}}=0.
\end{align*}
\end{proof}
Arguing as in the previous lemma we can prove the next result.
% (see \cite{A3, ZZR} when $p=2$): 
%and Lemma $3.4$ in \cite{ZZR} when $p=2$):
\begin{lem}\label{psilem}
Let $(u_{n})\subset D^{s,p}(\R^{N})$ be a bounded sequence and $\psi\in C^{\infty}_{c}(\R^{N})$ be such that $0\leq \psi\leq 1$, $\psi=1$ in $B_{1}$, $\psi=0$ in $\R^{N}\setminus B_{2}$ and $|\nabla \psi|_{\infty}\leq 2$. Set $\psi_{\rho}(x)=\psi(\frac{x-x_{i}}{\rho})$ where $x_{i}\in \R^{N}$ is a fixed point. Then we have
\begin{equation}\label{NIOOcritico}
\lim_{\rho\rightarrow 0}\limsup_{n\rightarrow \infty}  \iint_{\R^{2N}} |u_{n}(x)|^{p} \frac{|\psi_{\rho}(x)-\psi_{\rho}(y)|^{p}}{|x-y|^{N+sp}} \, dx dy =0.
\end{equation}
\end{lem}
\begin{proof}
For the reader convenience, we give the details of the proof. It is clear that
\begin{align*}
\R^{2N}&=((\R^{N}\setminus B_{2\rho}(x_{i}))\times (\R^{N}\setminus B_{2\rho}(x_{i})))\cup (B_{2\rho}(x_{i})\times \R^{N}) \cup ((\R^{N}\setminus B_{2\rho}(x_{i}))\times B_{2\rho}(x_{i}))\\
&=: X^{1}_{\rho}\cup X^{2}_{\rho} \cup X^{3}_{\rho}.
\end{align*}
Hence,
\begin{align}\label{Pa1critico}
&\iint_{\R^{2N}}  |u_{n}(x)|^{p} \frac{|\psi_{\rho}(x)-\psi_{\rho}(y)|^{p}}{|x-y|^{N+sp}} \, dx dy \nonumber\\
&=\iint_{X^{1}_{\rho}}  |u_{n}(x)|^{p}\frac{|\psi_{\rho}(x)-\psi_{\rho}(y)|^{p}}{|x-y|^{N+sp}} \, dx dy +\iint_{X^{2}_{\rho}}  |u_{n}(x)|^{2}\frac{|\psi_{\rho}(x)-\psi_{\rho}(y)|^{p}}{|x-y|^{N+sp}} \, dx dy \nonumber\\
&+ \iint_{X^{3}_{\rho}}  |u_{n}(x)|^{p}\frac{|\psi_{\rho}(x)-\psi_{\rho}(y)|^{p}}{|x-y|^{N+sp}} \, dx dy.
\end{align}
In what follows, we estimate each integral in (\ref{Pa1critico}).
Since $\psi=0$ in $\R^{N}\setminus B_{2}$, we have
\begin{align}\label{Pa2critico}
\iint_{X^{1}_{\rho}}  |u_{n}(x)|^{p}\frac{|\psi_{\rho}(x)-\psi_{\rho}(y)|^{p}}{|x-y|^{N+sp}} \, dx dy=0.
\end{align}
Recalling that $0\leq \psi\leq 1$ and $|\nabla \psi|_{\infty}\leq C$, we obtain
\begin{align}\label{Pa3critico}
&\iint_{X^{2}_{\rho}}  |u_{n}(x)|^{p}\frac{|\psi_{\rho}(x)-\psi_{\rho}(y)|^{p}}{|x-y|^{N+sp}} \, dx dy \nonumber\\
&=\int_{B_{2\rho}(x_{i})} \,dx \int_{\{y\in \R^{N}: |x-y|\leq \rho\}}  |u_{n}(x)|^{p}\frac{|\psi_{\rho}(x)-\psi_{\rho}(y)|^{p}}{|x-y|^{N+2s}} \, dy \nonumber \\
&\quad+\int_{B_{2\rho}(x_{i})} \, dx \int_{\{y\in \R^{N}: |x-y|> \rho\}} |u_{n}(x)|^{p}\frac{|\psi_{\rho}(x)-\psi_{\rho}(y)|^{p}}{|x-y|^{N+sp}} \, dy  \nonumber\\
&\leq C\rho^{-p} \int_{B_{2\rho}(x_{i})} \, dx \int_{\{y\in \R^{N}: |x-y|\leq \rho\}} \frac{ |u_{n}(x)|^{2}}{|x-y|^{N+ps-p}} \, dy\nonumber \\
&\quad+ C \int_{B_{2\rho}(x_{i})} \, dx \int_{\{y\in \R^{N}: |x-y|> \rho\}} \frac{ |u_{n}(x)|^{p}}{|x-y|^{N+sp}} \, dy \nonumber\\
&\leq C \rho^{-sp} \int_{B_{2\rho}(x_{i})}  |u_{n}(x)|^{p} \, dx+C \rho^{-sp} \int_{B_{2\rho}(x_{i})}  |u_{n}(x)|^{p} \, dx \nonumber \\
&\leq C \rho^{-sp} \int_{B_{2\rho}(x_{i})}  |u_{n}(x)|^{p} \, dx.
\end{align}
On the other hand,
\begin{align}\label{Pa4critico}
&\iint_{X^{3}_{\rho}}  |u_{n}(x)|^{p}\frac{|\psi_{\rho}(x)-\psi_{\rho}(y)|^{2}}{|x-y|^{N+sp}} \, dx dy \nonumber\\
&=\int_{\R^{N}\setminus B_{2\rho}(x_{i})} \, dx \int_{\{y\in B_{2\rho}(x_{i}): |x-y|\leq \rho\}}  |u_{n}(x)|^{p}\frac{|\psi_{\rho}(x)-\psi_{\rho}(y)|^{p}}{|x-y|^{N+sp}} \, dy \nonumber\\
&\quad+\int_{\R^{N}\setminus B_{2\rho}(x_{i})} \,dx \int_{\{y\in B_{2\rho}(x_{i}): |x-y|> \rho\}}  |u_{n}(x)|^{p}\frac{|\psi_{\rho}(x)-\psi_{\rho}(y)|^{p}}{|x-y|^{N+sp}} \, dy\nonumber \\
&=: A_{\rho, n}+ B_{\rho, n}. 
\end{align}
Let us note that, if $|x-y|<\rho$ and $|y-x_{i}|<2\rho$, then $|x-x_{i}|<3\rho$. Accordingly,
\begin{align}\label{Pa5critico}
A_{\rho, n}&\leq \rho^{-p} |\nabla \psi|_{\infty}^{p} \int_{B_{3\rho}(x_{i})} \, dx \int_{\{y\in B_{2\rho}(x_{i}): |x-y|\leq \rho\}} \frac{ |u_{n}(x)|^{p}}{|x-y|^{N+ps-p}} \, dy \nonumber\\
&\leq C\rho^{-p}   \int_{B_{3\rho}(x_{i})}  |u_{n}(x)|^{p} \, dx \int_{\{z\in \R^{N}: |z|\leq \rho\}} \frac{1}{|z|^{N+sp-p}} \, dz \nonumber\\
&\leq C \rho^{-sp} \int_{B_{3\rho}(x_{i})}  |u_{n}(x)|^{p} \, dx.
\end{align}
Now, for all $K>4$, it holds 
$$
(\R^{N}\setminus B_{2\rho}(x_{i}))\times B_{2\rho}(x_{i}) \subset (B_{K\rho}(x_{i})\times B_{2\rho}(x_{i})) \cup ((\R^{N}\setminus B_{K\rho}(x_{i}))\times B_{2\rho}(x_{i})).
$$
Therefore, 
\begin{align}\label{Pa6critico}
&\int_{B_{K\rho}(x_{i})} \, dx \int_{\{y\in B_{2\rho}(x_{i}): |x-y|> \rho\}}  |u_{n}(x)|^{p}\frac{|\psi_{\rho}(x)-\psi_{\rho}(y)|^{p}}{|x-y|^{N+sp}} \, dy \nonumber\\
&\leq C\int_{B_{K\rho}(x_{i})} \, dx \int_{\{y\in B_{2\rho}(x_{i}): |x-y|> \rho\}}  |u_{n}(x)|^{p} \frac{1}{|x-y|^{N+sp}} \, dy \nonumber \\
&\leq C  \int_{B_{K\rho}(x_{i})} \, dx \int_{\{z\in \R^{N}: |z|> \rho\}}  |u_{n}(x)|^{p} \frac{1}{|z|^{N+sp}} \, dz \nonumber\\
&\leq C \rho^{-sp} \int_{B_{K\rho}(x_{i})}  |u_{n}(x)|^{p} \, dx.
\end{align}
On the other hand, if $|x-x_{i}|\geq K\rho$ and $|y-x_{i}|<2\rho$ then 
$$
|x-y|\geq |x-x_{i}|-|y-x_{i}|\geq \frac{|x-x_{i}|}{2}+\frac{K\rho}{2}-2\rho>\frac{|x-x_{i}|}{2}.
$$
Consequently, using the H\"older inequality we have
\begin{align}\label{Pa7critico}
&\int_{\R^{N}\setminus B_{K\rho}(x_{i})} \, dx \int_{\{y\in B_{2\rho}(x_{i}): |x-y|>\rho\}}  |u_{n}(x)|^{p} \frac{|\psi_{\rho}(x)-\psi_{\rho}(y)|^{p}}{|x-y|^{N+sp}} \, dy \nonumber\\
&\leq  C \int_{\R^{N}\setminus B_{K\rho}(x_{i})} \, dx \int_{\{y\in B_{2\rho}(x_{i}): |x-y|>\rho\}} \frac{ |u_{n}(x)|^{p}}{|x-x_{i}|^{N+sp}} \, dy \nonumber\\
&\leq C\rho^{N} \int_{\R^{N}\setminus B_{K\rho}(x_{i})} \frac{ |u_{n}(x)|^{p}}{|x-x_{i}|^{N+sp}} \, dx \nonumber\\
&\leq C \rho^{N}\!\left(\int_{\R^{N}\setminus B_{K\rho}(x_{i})} \!\!\! |u_{n}(x)|^{\2} \, dx\right)^{\frac{p}{\2}}\!\! \left(\int_{\R^{N}\setminus B_{K\rho}(x_{i})} \!\!\!|x-x_{i}|^{-(N+sp)\frac{\2}{\2-p}} \, dx\right)^{\frac{\2-p}{\2}} \nonumber\\
&\leq C K^{-N} \left(\int_{\R^{N}\setminus B_{K\rho}(x_{i})}  |u_{n}(x)|^{\2} \, dx\right)^{\frac{p}{p^{*}_{s}}}.
\end{align}
Putting together (\ref{Pa6critico}) and (\ref{Pa7critico}), and using the fact that $(u_{n})$ is bounded in $L^{p^{*}_{s}}(\R^{N})$, we obtain that 
\begin{align}\label{Pa8critico}
B_{\rho, n}\leq C \rho^{-sp} \int_{B_{K\rho}(x_{i})}  |u_{n}(x)|^{p} \, dx+C K^{-N}.
\end{align}
Then, (\ref{Pa1critico})-(\ref{Pa5critico}) and (\ref{Pa8critico}) yield
\begin{align}\label{stimacritico}
\iint_{\R^{2N}}  |u_{n}(x)|^{p}\frac{|\psi_{\rho}(x)-\psi_{\rho}(y)|^{p}}{|x-y|^{N+sp}} \, dx dy \leq C \rho^{-sp} \int_{B_{K\rho}(x_{i})}  |u_{n}(x)|^{p} \, dx+C K^{-N}.
\end{align}
Recalling that $u_{n}\rightarrow u$ strongly in $L^{p}_{loc}(\R^{N}, \R)$ we get
\begin{align*}
\lim_{n\rightarrow \infty}  C \rho^{-sp} \int_{B_{K\rho}(x_{i})}  |u_{n}(x)|^{p} \, dx+C K^{-N} =C \rho^{-sp} \int_{B_{K\rho}(x_{i})} |u(x)|^{p} \, dx+C K^{-N}.
\end{align*}
Using the H\"older inequality we can see that
\begin{align*}
C \rho^{-sp} & \int_{B_{K\rho}(x_{i})}  |u(x)|^{p} \, dx+C K^{-N} \\
&\leq C \rho^{-sp} \left(\int_{B_{K\rho}(x_{i})}  |u(x)|^{p^{*}_{s}} \, dx\right)^{\frac{p}{p^{*}_{s}}} |B_{K\rho}(x_{i})|^{1-\frac{p}{p^{*}_{s}}}+C K^{-N} \\
&\leq C K^{sp}  \left(\int_{B_{K\rho}(x_{i})}  |u(x)|^{p^{*}_{s}} \, dx\right)^{\frac{p}{p^{*}_{s}}}+C K^{-N}\rightarrow C K^{-N} \mbox{ as } \rho\rightarrow 0.
\end{align*}
In conclusion, 
\begin{align*}
&\lim_{\rho\rightarrow 0} \limsup_{n\rightarrow \infty} \iint_{\R^{2N}}  |u_{n}(x)|^{p}\frac{|\psi_{\rho}(x)-\psi_{\rho}(y)|^{p}}{|x-y|^{N+sp}} \, dx dy \nonumber\\
&=\lim_{K\rightarrow \infty}\lim_{\rho\rightarrow 0} \limsup_{n\rightarrow \infty} \iint_{\R^{2N}}  |u_{n}(x)|^{p}\frac{|\psi_{\rho}(x)-\psi_{\rho}(y)|^{p}}{|x-y|^{N+sp}} \, dx dy =0.
\end{align*}
\end{proof}

\noindent
We now give the proof of the following variant of the concentration-compactness lemma \cite{Lions} which is inspired by \cite{Ch, W}. We refer to \cite{A3, DPMV, PP, ZZR} for some results in the fractional context $p=2$.
In what follows, we will use the following notation 
$$
|D^{s} u|^{p}(x)=\int_{\R^{N}} \frac{|u(x)-u(y)|^{p}}{|x-y|^{N+sp}}dy.
$$
\begin{lem}\label{CCL}
Let $(u_{n})$ be a sequence in $D^{s,p}(\R^{N})$ such that $u_{n}\rightharpoonup u$ in $D^{s,p}(\R^{N})$. Let us assume that
\begin{align}\begin{split}\label{46FS}
&|D^{s} u_{n}|^{p}\rightharpoonup \mu \\
&|u_{n}|^{p^{*}_{s}}\rightharpoonup \nu 
\end{split}\end{align}
in the sense of measure, where $\mu$ and $\nu$ are two non-negative bounded measures on $\R^{N}$. 
Then, there exist an at most a countable set $I$, a family of distinct points $(x_{i})_{i\in I}\subset \R^{N}$ and $(\mu_{i})_{i\in I}, (\nu_{i})_{i\in I}\subset (0, \infty)$ such that
\begin{align}
&\nu=|u|^{p^{*}_{s}}+\sum_{i\in I} \nu_{i} \delta_{x_{i}} \label{47FS},\\
&\mu\geq |D^{s}u|^{p}+\sum_{i\in I} \mu_{i} \delta_{x_{i}} \label{48FS}, \\
&\mu_{i}\geq S_{*} \nu_{i}^{\frac{p}{p^{*}_{s}}} \quad \forall i\in I  \label{50FS}.
\end{align}
Moreover, if we define
\begin{equation}\label{muinftydef}
\mu_{\infty}=\lim_{R\rightarrow \infty} \limsup_{n\rightarrow \infty} \int_{|x|>R} |D^{s}u_{n}|^{p}dx,
\end{equation}
and
\begin{equation}\label{nuinftydef}
\nu_{\infty}=\lim_{R\rightarrow \infty} \limsup_{n\rightarrow \infty} \int_{|x|>R} |u_{n}|^{p^{*}_{s}}dx,
\end{equation}
then
\begin{align}
&\limsup_{n\rightarrow \infty} \int_{\R^{N}} |D^{s}u_{n}|^{p}dx=\mu(\R^{N})+\mu_{\infty} \label{51FS},\\
&\limsup_{n\rightarrow \infty} \int_{\R^{N}} |u_{n}|^{p^{*}_{s}}dx=\nu(\R^{N})+\nu_{\infty} \label{52FS}, \\
&\mu_{\infty}\geq S_{*} \nu_{\infty}^{\frac{p}{p^{*}_{s}}} \label{53FS}.
\end{align}
\end{lem}
\begin{proof}
In order to prove \eqref{47FS}, we aim to pass to the limit in the following relation which holds in view of the Brezis-Lieb lemma \cite{BL}:
\begin{align}\label{55fs}
&\int_{\R^{N}} |\psi|^{p^{*}_{s}} |u_{n}|^{p^{*}_{s}}\, dx = \int_{\R^{N}} |\psi|^{p^{*}_{s}} |u|^{p^{*}_{s}} \, dx+ \int_{\R^{N}} |\psi|^{p^{*}_{s}} |u_{n}-u|^{p^{*}_{s}} \, dx + o_{n}(1), 
\end{align}
where $\psi\in C^{\infty}_{c}(\R^{N})$. Set $\tilde{u}_{n}=u_{n}-u$. Then, by Theorem \ref{Sembedding}, $\tilde{u}_{n}\rightarrow 0$ in $L^{p}_{loc}(\R^{N})$ and a.e. in $\R^{N}$.
Fix $\psi\in C^{\infty}_{c}(\R^{N})$. Using the definition of $S_{*}$, we have
\begin{align}\label{56fs}
\left[\int_{\R^{N}}  |\psi|^{p^{*}_{s}} |u_{n}-u|^{p^{*}_{s}} \, dx\right]^{\frac{p}{p^{*}_{s}}} &= \left[\int_{\R^{N}} |\psi \tilde{u}_{n}|^{p^{*}_{s}} \, dx\right]^{\frac{p}{p^{*}_{s}}}\nonumber \\
&\leq S^{-1}_{*} \int_{\R^{N}} (|D^{s} (\psi\, \tilde{u}_{n})|^{p} dx\nonumber \\
&=S^{-1}_{*} \left[\iint_{\R^{2N}} \frac{|(\psi\tilde{u}_{n})(x)-(\psi \tilde{u}_{n})(y)|^{p}}{|x-y|^{N+sp}} dx dy \right].
\end{align}
Now, we observe that
\begin{align*}
&\iint_{\R^{2N}} \frac{|(\psi\tilde{u}_{n})(x)-(\psi \tilde{u}_{n})(y)|^{p}}{|x-y|^{N+sp}} dx dy \nonumber\\
&\leq 2^{p-1} \Biggl(\iint_{\R^{2N}}  |\psi(y)|^{p} \frac{|\tilde{u}_{n}(x)- \tilde{u}_{n}(y)|^{p}}{|x-y|^{N+sp}}  +|\tilde{u}_{n}(x)|^{p} \frac{|\psi(x)- \psi(y)|^{p}}{|x-y|^{N+sp}} \, dx dy\Biggr).
\end{align*}
It is easy to show that
\begin{align*}
\iint_{\R^{2N}} \frac{|\psi(x)- \psi(y)|^{p}}{|x-y|^{N+2s}} |\tilde{u}_{n}(x)|^{p} dxdy = o_{n}(1). 
\end{align*}
Indeed, arguing as in the proof of Lemma \ref{psilem} (with $x_{i}=0$ and $\rho=1$), if $\psi=1$ in $B_{1}$ and $\psi=0$ in $\R^{N}\setminus B_{2}$, we have
\begin{align*}
&\int_{\R^{N}} \,\int_{\R^{N}} \frac{|\psi(x)- \psi(y)|^{p}}{|x-y|^{N+sp}} |\tilde{u}_{n}(x)|^{p}\, dx dy \\
&=\int_{B_{2}} \int_{\R^{N}} \frac{|\psi(x)- \psi(y)|^{p}}{|x-y|^{N+sp}} |\tilde{u}_{n}(x)|^{p} dx dy+\int_{\R^{N}\setminus B_{2}} \int_{B_{2}} \frac{|\psi(x)- \psi(y)|^{p}}{|x-y|^{N+sp}} |\tilde{u}_{n}(x)|^{p}dx dy \\
&\leq C\int_{B_{K}}  |\tilde{u}_{n}(x)|^{p}\, dx +CK^{-N} \quad  \forall K>4,
\end{align*}
and taking first the limit as $n\rightarrow \infty$ and then as $K\rightarrow \infty$ we get the desired conclusion. \\
Therefore, if we assume that $|D^{s} \tilde{u}_{n}|^{p}\rightharpoonup \tilde{\mu}$ and $|\tilde{u}_{n}|^{p^{*}_{s}}\rightharpoonup \tilde{\nu}$ in the sense of measures, from the above facts and by passing to the limit in \eqref{56fs}, we deduce that
\begin{align*}
\left[ \int_{\R^{N}} |\psi|^{p^{*}_{s}} d\tilde{\nu}\right]^{\frac{1}{p^{*}_{s}}} \leq C \left[ \int_{\R^{N}} |\psi|^{p} d\tilde{\mu}\right]^{\frac{1}{p^{*}_{s}}}, \, \mbox{ for all } \psi \in C^{\infty}_{c}(\R^{N}). 
\end{align*}
Then, using Lemma $1.2$ in \cite{Lions}, there exist at most a countable set $I$, families $(x_{i})_{i\in I}\subset \R^{N}$ and $(\nu_{i})_{i\in I}\subset (0, \infty)$ such that
\begin{align}\label{57fs} 
\tilde{\nu} = \sum_{i\in I} \nu_{i} \delta_{x_{i}}.
\end{align}
In view of \eqref{55fs}, we deduce that $\nu= |u|^{p^{*}_{s}} + \tilde{\nu}$ which together with \eqref{57fs} implies that 
$$
\nu= |u|^{p^{*}_{s}}+ \sum_{i\in I} \nu_{i} \delta_{x_{i}},
$$
that is \eqref{47FS} is verified. \\
Now, we prove that \eqref{50FS} holds true. Take $\psi_{\rho}= \eta(\frac{x-x_{i}}{\rho})$, where $\eta \in C^{\infty}_{c}(\R^{N})$, $0\leq \eta \leq 1$, $\eta=1$ in $B_{1}$ and $\eta=0$ in $\R^{N}\setminus B_{2}$. 
Then, recalling the definition of $S_{*}$ and the following inequality
\begin{align}\label{xy}
(x+y)^{p}\leq x^{p} + C_{p} y^{p}, \, \mbox{ for all } x, y\geq 0, p>1,
\end{align}
we obtain
\begin{align}\label{58fs}
S_{*} \left[ \int_{\R^{N}} |\psi_{\rho}|^{p^{*}_{s}} |u_{n}|^{p^{*}_{s}}\, dx \right]^{\frac{p}{p^{*}_{s}}}&\leq \int_{\R^{N}} |D^{s} (\psi_{\rho} \, u_{n})|^{p} dx \nonumber \\
&\leq  C_{p} \Biggl( \iint_{\R^{2N}} |u_{n}(x)|^{p} \frac{|\psi_{\rho}(x)- \psi_{\rho}(y)|^{p}}{|x-y|^{N+sp}} dxdy\Biggr) \nonumber\\
&+ \Biggl(\iint_{\R^{2N}} |\psi_{\rho}(y)|^{p} \frac{|u_{n}(x)- u_{n}(y)|^{p}}{|x-y|^{N+sp}} dxdy\Biggr).
\end{align}
Now, taking into account \eqref{46FS} and \eqref{47FS}, we have
\begin{align*}
\lim_{n\rightarrow \infty}\int_{\R^{N}} |\psi_{\rho}|^{p^{*}_{s}} |u_{n}|^{p^{*}_{s}}\, dx=\int_{B_{\rho}(x_{j})} |\psi_{\rho}|^{p^{*}_{s}} |u|^{p^{*}_{s}}\, dx + \nu_{i}.
\end{align*}
Since $0\leq\psi_{\rho}\leq 1$ implies
\begin{equation*}
\left|\int_{B_{\rho}(x_{j})} |\psi_{\rho}|^{p^{*}_{s}} |u|^{p^{*}_{s}}\, dx\right|\leq C\int_{B_{\rho}(x_{j})} |u|^{p^{*}_{s}}dx\rightarrow 0 \mbox{ as } \rho\rightarrow 0,
\end{equation*}
we deduce that
\begin{align}\label{stef1}
\lim_{\rho\rightarrow 0}\lim_{n\rightarrow \infty}\int_{\R^{N}} |\psi_{\rho}|^{p^{*}_{s}} |u_{n}|^{p^{*}_{s}} dx= \nu_{i}.
\end{align}
On the other hand, \eqref{46FS} gives 
\begin{equation*}
\lim_{n\rightarrow \infty}\iint_{\R^{2N}} |\psi_{\rho}(y)|^{p} \frac{|u_{n}(x)- u_{n}(y)|^{p}}{|x-y|^{N+sp}}=\int_{\R^{N}} |\psi_{\rho}(y)|^{p} \, d\mu,
\end{equation*}
and using Lemma \ref{psilem} we can see that
\begin{align}\label{stef3}
\lim_{\rho\rightarrow 0}\limsup_{n\rightarrow \infty} &\iint_{\R^{2N}} |u_{n}(x)|^{p} \frac{|\psi_{\rho}(x)- \psi_{\rho}(y)|^{p}}{|x-y|^{N+sp}} \, dx dy=0.
\end{align}
Then, putting together  \eqref{58fs}, \eqref{stef1} and \eqref{stef3} we get
\begin{align*}
S_{*} \nu_{i}^{\frac{p}{p^{*}_{s}}} \leq \lim_{\rho\rightarrow 0} \mu(B_{\rho}(x_{i})). 
\end{align*}
Setting $\mu_{i}= \lim_{\rho\rightarrow 0} \mu(B_{\rho}(x_{i}))$ we deduce that \eqref{50FS} holds true. \\
Now, we can note that
\begin{align*}
\mu\geq \sum_{i\in I} \mu_{i} \delta_{x_{i}}
\end{align*}
and that the weak convergences implies that $\mu \geq |D^{s}u|^{p}$. Then, due to the fact that $|D^{s}u|^{p}$ is orthogonal to $\sum_{i\in I} \mu_{i} \delta_{x_{i}}$, we can infer that \eqref{48FS} is satisfied. 
Finally, we show the validity of \eqref{51FS}-\eqref{53FS}. 
Let $\eta_{R}$ be defined as in Lemma \ref{etaR}. Then we have
\begin{align}\label{BS1}
\int_{\R^{N}} |D^{s}u_{n}|^{p}dx=\int_{\R^{N}} |D^{s}u_{n}|^{p}\eta_{R} dx+\int_{\R^{N}} |D^{s}u_{n}|^{p}(1-\eta_{R})dx.
\end{align}
Since
$$
\int_{|x|>2R} |D^{s}u_{n}|^{p}dx\leq \int_{\R^{N}} |D^{s}u_{n}|^{p}\eta_{R} dx\leq \int_{|x|>R} |D^{s}u_{n}|^{p}dx,
$$
we obtain that 
\begin{equation}\label{BS2}
\mu_{\infty}=\lim_{R\rightarrow \infty} \limsup_{n\rightarrow \infty} \int_{\R^{N}} |D^{s}u_{n}|^{p} \eta_{R} dx.
\end{equation}
Now, using that $\mu$ is finite and $1-\eta_{R}$ has a compact support, we can apply the dominated convergence theorem to get
\begin{align}\label{BS3}
\lim_{R\rightarrow \infty} \limsup_{n\rightarrow \infty} \int_{\R^{N}} |D^{s}u_{n}|^{p}(1-\eta_{R})dx=\lim_{R\rightarrow \infty}  \int_{\R^{N}} (1-\eta_{R})d\mu=\mu(\R^{N}).
\end{align}
Gathering \eqref{BS1}, \eqref{BS2} and \eqref{BS3} we obtain that \eqref{51FS} holds true. In a similar fashion, we can prove that 
\begin{equation}\label{BS4}
\nu_{\infty}=\lim_{R\rightarrow \infty} \limsup_{n\rightarrow \infty} \int_{\R^{N}} |u_{n}|^{p^{*}_{s}} \eta_{R} dx,
\end{equation}
and arguing as before we deduce that \eqref{52FS} is verified. In order to show that \eqref{53FS} is satisfied, we can use Theorem \ref{Sembedding}, $0\leq \eta_{R}\leq 1$ and \eqref{xy}
 to see that
\begin{align}\label{HMI}
S_{*}\left[\int_{\R^{N}} |\eta_{R}u_{n}|^{p^{*}_{s}} dx\right]^{\frac{p}{p^{*}_{s}}}&\leq \int_{\R^{N}} |D^{s}(\eta_{R}u_{n})|^{p}dx   \nonumber\\
&\leq \int_{\R^{N}} \eta_{R} |D^{s}u_{n}|^{p}dx+C_{p}\int_{\R^{N}} |u_{n}|^{p}|D^{s}\eta_{R}|^{p}dx.
\end{align} 
Then, by Lemma \ref{etaR}, we know that
$$
\lim_{R\rightarrow \infty}\limsup_{n\rightarrow \infty}\int_{\R^{N}} |u_{n}|^{p}|D^{s}\eta_{R}|^{p}dx=0,
$$
which together with  \eqref{BS2}, \eqref{BS4}  and \eqref{HMI} yields \eqref{53FS}. This ends the proof of lemma.
\end{proof}

\section{proof of Theorem \ref{thm1}}
This section is devoted to the proof of Theorem \ref{thm1}.
Let us recall that the functional $J: X\rightarrow \R$ associated with problem \eqref{P} is given by
$$
J(u)=\frac{1}{p}\|u\|^{p}_{s,p}+\frac{1}{q}\|u\|^{q}_{s,q}-\lambda\int_{\R^{N}} h(x) F(u)dx-\frac{1}{q^{*}_{s}}|u|_{q^{*}_{s}}^{q^{*}_{s}}.
$$
From assumptions $(f_1)$-$(f_2)$, we know that for all $\e>0$ there exists $C_{\e}>0$ such that
\begin{equation}\label{f}
|f(t)|\leq \e|t|^{p-1}+C_{\e}|t|^{r-1} \mbox{ for all } t\in \R. 
\end{equation}
Using \eqref{f} and Theorem \ref{Sembedding}, it is easy to check that $J$ is well-defined on $X$ and $J\in C^{1}(X, \R)$. 
Now, we prove that $J$ possesses a mountain pass geometry \cite{AR}:
\begin{lem}\label{lem1}
For each $\lambda>0$ the functional $J$ satisfies the following conditions:
\begin{compactenum}[$(i)$]
\item there exist $\alpha, \beta>0$ such that $J(u)\geq \beta$ if $\|u\|=\alpha$,
\item there exists $e\in X$ such that $\|e\|>\alpha$ and $J(e)<0$.
\end{compactenum}
\end{lem}
\begin{proof}
By \eqref{f}, it follows that for all $\e\in (0,1)$ there exists $C_{\e}>0$ such that
\begin{align*}
J(u)&\geq \frac{1}{p}\|u\|^{p}_{s,p}+\frac{1}{q}\|u\|^{q}_{s,q}-\lambda \frac{\e}{p}|h|_{\infty}|u|_{p}^{p}-\lambda |h|_{\infty}\frac{C_{\e}}{r} |u|^{r}_{r}-\frac{1}{\q}|u|_{\q}^{\q} \\
&\geq \frac{1}{p}([u]_{s,p}^{p}+(1-\e)|h|_{\infty}|u|_{p}^{p})+ \frac{1}{q}([u]_{s,q}^{q}+|u|_{q}^{q})-\lambda |h|_{\infty}\frac{C_{\e}}{r} |u|^{r}_{r}-\frac{1}{\q}|u|_{\q}^{\q} \\
&\geq C_{1}(\|u\|^{p}_{s,p}+\|u\|^{q}_{s,q}) -\lambda \frac{C_{\e}}{r} |h|_{\infty} |u|^{r}_{r}-\frac{1}{\q}|u|_{\q}^{\q}, 
\end{align*}
where
$$
C_{1}=\min\left\{\frac{1}{p}\min\{1, (1-\e)|h|_{\infty}\}, \frac{1}{q}\right\}.
$$
Now, if $\|u\|<1$, by $q>p$ it follows that $\|u\|^{p}_{s,p}\geq \|u\|_{s,p}^{q}$, and consequently
\begin{align*}
J(u)&\geq C_{1}(\|u\|^{q}_{s,p}+\|u\|^{q}_{s,q}) -\lambda |h|_{\infty}\frac{C_{\e}}{r} |u|^{r}_{r}-\frac{1}{\q}|u|_{\q}^{\q} \\
&\geq C_{2} \|u\|^{q}-\lambda |h|_{\infty}\frac{C_{\e}}{r} |u|^{r}_{r}-\frac{1}{\q}|u|_{\q}^{\q}.
\end{align*}
Thus, it follows from Theorem \ref{Sembedding} that
\begin{align*}
J(u)&\geq C_{2} \|u\|^{q}-\lambda |h|_{\infty}\frac{C_{\e}}{r} \|u\|^{r}_{s,q}-C_{3}\|u\|_{s,q}^{\q} \\
&\geq C_{2}\|u\|^{q}-\lambda C_{4} \|u\|^{r}-C_{5} \|u\|^{\q}\\
&=\|u\|^{q}(C_{2}-\lambda C_{4}\|u\|^{r-q}-C_{5} \|u\|^{\q-q}).
\end{align*}
Since $r\in (q, \q)$, there exist $\alpha, \beta>0$ such that $J(u)\geq \beta$ for all $u\in X$ such that $\|u\|=\alpha$. Hence, $(i)$ is verified.\\
Fix $v\in C^{\infty}_{c}(\R^{N})$ such that $v\geq 0$ and $v\not\equiv 0$ in $\R^{N}$. 
We recall that $(f_3)$ implies that
\begin{equation}\label{AR}
F(t)\geq A t^{\theta}-B \mbox{ for all } t>1,
\end{equation}
for some $A, B>0$.
Hence, using \eqref{AR}, $(h)$ and $\theta\in (q, \q)$ we obtain that
\begin{align*}
J(t v)&\leq \frac{t^{p}}{p}\|v\|^{p}_{s,p}+\frac{t^{q}}{q}\|v\|^{q}_{s,q}-\lambda A t^{\theta} \int_{\supp(v)} v^{\theta} h \,dx+B|\supp(v)||h|_{\infty} \rightarrow -\infty
\end{align*}
as $t\rightarrow \infty$. Therefore, we can find $\tau>0$ sufficiently large such that $\|\tau v\|> \alpha$ and $J(\tau v)<0$. This fact shows that $(ii)$ holds true.
\end{proof}

\noindent
In view of Lemma \ref{lem1} we can define the mountain pass level
$$
c_{*}=\inf_{\gamma\in \Gamma}\max_{t\in [0, 1]} J(\gamma(t)),
$$
where
$$
\Gamma=\left\{\gamma\in C([0, 1], X): \gamma(0)=0 \,\mbox{ and } \, J(\gamma(1))<0\right\}.
$$
Our goal is to prove that $c_{*}$ is achieved by some nontrivial function $u\in X$.
Firstly, we show that it is possible to compare $c_{*}$ with a suitable constant which involves $S_{*}$:
\begin{lem}\label{lem2}
There exists $\lambda_{*}>0$ such that $c_{*}\in \left(0, \left(\frac{1}{\theta}-\frac{1}{q^{*}_{s}}\right) S_{*}^{\frac{N}{sq}}\right)$ for all $\lambda\geq \lambda_{*}$.
\end{lem}
\begin{proof}
Let $v\in C^{\infty}_{c}(\R^{N})$ be such that $v\geq 0$ and $v\not\equiv 0$ in $\R^{N}$. Then there exists $t_{\lambda}>0$ such that $J(t_{\lambda} v)=\max_{t\geq 0} J(t v)$.
Accordingly, $\langle J'(t_{\lambda} v), t_{\lambda} v\rangle=0$ that is
\begin{align}\label{nehari}
t_{\lambda}^{p}\|v\|^{p}_{s,p}+t_{\lambda}^{q}\|v\|^{q}_{s,q}=\lambda \int_{\R^{N}} h(x)f(t_{\lambda} v)t_{\lambda} v dx+t^{\q}_{\lambda} |v|_{\q}^{\q}
\end{align}
which combined with $(f_3)$ yields
\begin{align*}
t_{\lambda}^{p}\|v\|^{p}_{s,p}+t_{\lambda}^{q}\|v\|^{q}_{s,q}\geq t^{\q}_{\lambda} |v|_{\q}^{\q}.
\end{align*}
Since $p\leq q<\q$, we can infer that $t_{\lambda}$ is bounded and that there exists a sequence $\lambda_{n}\rightarrow \infty$ such that $t_{\lambda_{n}}\rightarrow t_{0}\geq 0$.
Let us observe that if $t_{0}>0$ then we have
\begin{align*}
t_{\lambda_{n}}^{p}\|v\|^{p}_{s,p}+t_{\lambda_{n}}^{q}\|v\|^{q}_{s,q}\rightarrow L\in (0, \infty)
\end{align*}
and
$$
\lambda_{n} \int_{\R^{N}} h(x)f(t_{\lambda_{n}} v)t_{\lambda_{n}} v dx+t^{\q}_{\lambda_{n}} |v|_{\q}^{\q}\rightarrow \infty
$$
which gives a contradiction in view of \eqref{nehari}. Therefore, $t_{0}=0$. Let us now define $\gamma(t)=t v$ with $t\in [0, 1]$. Then, $\gamma\in \Gamma$ and we get
\begin{equation}\label{c*}
0<c_{*}\leq \max_{t\in [0, 1]} J(t v)=J(t_{\lambda} v)\leq t_{\lambda}^{p}\|v\|^{p}_{s,p}+t_{\lambda}^{q}\|v\|^{q}_{s,q}.
\end{equation}
Taking $\lambda$ sufficiently large, we obtain that 
$$
t_{\lambda}^{p}\|v\|^{p}_{s,p}+t_{\lambda}^{q}\|v\|^{q}_{s,q}<\left(\frac{1}{\theta}-\frac{1}{q^{*}_{s}}\right) S_{*}^{\frac{N}{sq}},
$$
which yields
$$
0<c_{*}<\left(\frac{1}{\theta}-\frac{1}{q^{*}_{s}}\right) S_{*}^{\frac{N}{sq}}.
$$
Moreover, since $t_{\lambda}\rightarrow 0$ as $\lambda\rightarrow \infty$, it follows from \eqref{c*} that $c_{*}\rightarrow 0$ as $\lambda\rightarrow \infty$.
\end{proof}

\noindent
In the lemma below we will make use of the concentration-compactness lemma proved in Section $2$ to verify the Palais-Smale condition:
\begin{lem}\label{lem3}
Let $(u_{n})\subset X$ be a $(PS)_{c_{*}}$ sequence for $J$.
Then, up to subsequences, $u_{n}\rightarrow u$ in $X$ for all $\lambda\geq \lambda_{*}$.
\end{lem}
\begin{proof}
We begin by proving that $(u_{n})$ is bounded in $X$. Since $J(u_{n})\rightarrow c_{*}$ and $J'(u_{n})\rightarrow 0$ we have
\begin{align*}
C(1+\|u_{n}\|)&\geq J(u_{n})-\frac{1}{\theta}\langle J'(u_{n}), u_{n}\rangle \\
&=\frac{1}{p}\|u_{n}\|^{p}_{s,p}+\frac{1}{q}\|u_{n}\|^{q}_{s,q}-\lambda\int_{\R^{N}} h(x)F(u_{n})dx-\frac{1}{q^{*}_{s}} |u_{n}|_{q^{*}_{s}}^{q^{*}_{s}} \\
&-\frac{1}{\theta}\left[\|u_{n}\|^{p}_{s,p}+\|u_{n}\|^{q}_{s,q}-\lambda\int_{\R^{N}} h(x)f(u_{n})u_{n}dx-|u_{n}|_{q^{*}_{s}}^{q^{*}_{s}} \right] \\
&\geq \left(\frac{1}{q}-\frac{1}{\theta}\right)(\|u_{n}\|^{p}_{s,p}+\|u_{n}\|^{q}_{s,q})+\frac{\lambda}{\theta}\int_{\R^{N}} h(x)[f(u_{n})u_{n}-\theta F(u_{n})]dx\\
&+\left(\frac{1}{\theta}-\frac{1}{\q}\right) |u_{n}|_{q^{*}_{s}}^{q^{*}_{s}}.
%&\geq  \left(\frac{1}{q}-\frac{1}{\theta}\right)(\|u\|^{p}_{s,p}+\|u\|^{q}_{s,q}).
\end{align*}
Then, using $(f_3)$ and $h\geq 0$, we deduce that
\begin{align}\label{absurd}
C(1+\|u_{n}\|)&\geq  \left(\frac{1}{q}-\frac{1}{\theta}\right)(\|u_{n}\|^{p}_{s,p}+\|u_{n}\|^{q}_{s,q}).
\end{align}
Now, we assume by contradiction that $\|u_{n}\|\rightarrow \infty$ and distinguish the following three cases:\\
Case 1: $\|u_{n}\|_{s,p}\rightarrow \infty$ and $\|u_{n}\|_{s,q}\rightarrow \infty$. \\
Then, for $n$ big enough, and using $q>p$, it holds $\|u_{n}\|^{q-p}_{s,q}\geq 1$, that is $\|u_{n}\|^{q}_{s,q}\geq \|u_{n}\|^{p}_{s,q}$. In view of \eqref{absurd} and $(a+b)^{p}\leq C_{p} (a^{p}+b^{p})$ for all $a, b\geq 0$, we can deduce that
$$
C(1+\|u_{n}\|)\geq  \left(\frac{1}{q}-\frac{1}{\theta}\right)(\|u_{n}\|^{p}_{s,p}+\|u_{n}\|^{p}_{s,q})\geq  C_{p}^{-1}\left(\frac{1}{q}-\frac{1}{\theta}\right)(\|u_{n}\|_{s,p}+\|u_{n}\|_{s,q})^{p}=: C_{1}\|u_{n}\|^{p}
$$
which implies that $\|u_{n}\|$ is bounded, that is a contradiction. \\
Case 2: $\|u_{n}\|_{s,p}\rightarrow \infty$ and $\|u_{n}\|_{s,q}$ is bounded.\\
From \eqref{absurd}, we have
$$
C(1+\|u_{n}\|_{s,p}+\|u_{n}\|_{s,q})=C(1+\|u_{n}\|)\geq  \left(\frac{1}{q}-\frac{1}{\theta}\right)\|u_{n}\|^{p}_{s,p}
$$
which yields
$$
C\left(\frac{1}{\|u_{n}\|^{p}_{s, p}}+\frac{1}{\|u_{n}\|^{p-1}_{s,p}}+\frac{\|u_{n}\|_{s,q}}{\|u_{n}\|^{p}_{s,p}} \right)\geq  \left(\frac{1}{q}-\frac{1}{\theta}\right).
$$
Taking the limit as $n\rightarrow \infty$, we get $0\geq \left(\frac{1}{q}-\frac{1}{\theta}\right)>0$ that is a contradiction.\\
Case 3: $\|u_{n}\|_{s,p}$ is bounded and $\|u_{n}\|_{s,q}\rightarrow \infty$.\\
The proof is similar to the previous one.

Summing up, $(u_{n})$ is bounded in $X$.
Then,  up to a subsequence, we may assume that there exists $u\in X$ such that $u_{n}\rightharpoonup u$ in $X$ and $u_{n}\rightarrow  u$ in $L^{t}_{loc}(\R^{N})$ for all $t\in [1, \q)$.
At this point, we prove that $\langle J'(u), \phi\rangle=0$ for all $\phi\in X$.
Consider the sequence
$$
U_{n}(x,y)=\frac{|u_{n}(x)-u_{n}(y)|^{p-2}(u_{n}(x)-u_{n}(y))}{|x-y|^{\frac{N+sp}{p'}}},
$$
and let
$$
U(x,y)=\frac{|u(x)-u(y)|^{p-2}(u(x)-u(y))}{|x-y|^{\frac{N+sp}{p'}}},
$$
where $p'=\frac{p}{p-1}$. It is easy to check that $(U_{n})$ is a bounded sequence in $L^{p'}(\R^{2N})$ and $U_{n}\rightarrow U$ a.e. in $\R^{2N}$. Since $L^{p'}(\R^{2N})$ is a reflexive space, there exists a subsequence, still denoted by $(U_{n})$, such that $U_{n}\rightharpoonup U$ in $L^{p'}(\R^{2N})$, that is
$$
\iint_{\R^{2N}} U_{n}(x,y) g(x,y) dx dy\rightarrow \iint_{\R^{2N}} U(x,y) g(x,y) dx dy \quad \forall g\in L^{p}(\R^{2N}).
$$
Then, for any $\phi\in C^{\infty}_{c}(\R^{N})$, we know that 
$$
g(x,y)=\frac{(\phi(x)-\phi(y))}{|x-y|^{\frac{N+sp}{p}}}\in L^{p}(\R^{2N}),
$$
and 
\begin{align*}
\iint_{\R^{2N}} &\frac{|u_{n}(x)-u_{n}(y)|^{p-2}(u_{n}(x)-u_{n}(y))(\phi(x)-\phi(y))}{|x-y|^{N+sp}} dx dy \\
&\rightarrow \iint_{\R^{2N}} \frac{|u(x)-u(y)|^{p-2}(u(x)-u(y))(\phi(x)-\phi(y))}{|x-y|^{N+sp}} dx dy.
\end{align*}
In a similar way, we can prove that 
\begin{align*}
\iint_{\R^{2N}} &\frac{|u_{n}(x)-u_{n}(y)|^{q-2}(u_{n}(x)-u_{n}(y))(\phi(x)-\phi(y))}{|x-y|^{N+sq}} dx dy \\
&\rightarrow \iint_{\R^{2N}} \frac{|u(x)-u(y)|^{q-2}(u(x)-u(y))(\phi(x)-\phi(y))}{|x-y|^{N+sq}} dx dy.
\end{align*}
On the other hand, since $f$ has subcritical growth and $h\in L^{\infty}(\R^{N})$, we have
$$
\int_{\R^{N}} h(x)f(u_{n})\phi dx\rightarrow \int_{\R^{N}} h(x) f(u)\phi dx, 
$$
$$
\int_{\R^{N}} |u_{n}|^{\q-2}u_{n}\phi dx \rightarrow \int_{\R^{N}} |u|^{\q-2}u\phi dx.
$$
Then, using the above limits and that $\langle J'(u_{n}),\phi \rangle=o_{n}(1)$, we obtain that $\langle J'(u),\phi \rangle=0$ for all $\phi\in C^{\infty}_{c}(\R^{N})$. From the density of $C^{\infty}_{c}(\R^{N})$ in $W^{s, p}(\R^{N})$, we can conclude that $u$ is a critical point of $J$.
Now, $(|D^{s}u_{n}|^{q})$ and $(|u_{n}|^{q^{*}_{s}})$ are bounded sequences in $L^{1}(\R^{N})$, so, applying Prokhorov's Theorem, up to subsequence, we may find two non-negative bounded measures $\mu$ and $\nu$ on $\R^{N}$ such that
\begin{align}\label{FIG}
|D^{s}u_{n}|^{q}\rightharpoonup \mu \mbox{ and } |u_{n}|^{\q}\rightharpoonup \nu.
\end{align}
In the light of Lemma \ref{CCL}, there exist an at most countable index set $I$, sequences $(x_i)\subset \R^{N}$, $(\mu_{i}), (\nu_{i})\subset (0, \infty)$ such that 
\begin{equation}\label{9}
\nu=|u|^{\q}+\sum_{i\in I} \nu_{i}\delta_{x_{i}}, \quad \mu\geq |D^{s}u|^{q}+\sum_{i\in I} \mu_{i}\delta_{x_{i}}, \quad S_{*}\nu_{i}^{q/\q}\leq \mu_{i} \quad \forall i\in I,
\end{equation}
and
\begin{align}
%&\limsup_{n\rightarrow \infty} \int_{|x|>R} |D^{s}u_{n}|^{q}dx=\mu(\R^{N})+\mu_{\infty} \label{91},\\
%&\limsup_{n\rightarrow \infty} \int_{|x|>R} |u_{n}|^{q^{*}_{s}}dx=\nu(\R^{N})+\nu_{\infty} \label{92}, \\
&S_{*}\nu_{\infty}^{q/\q}\leq \mu_{\infty} \label{93}.
\end{align}
where $\mu_{\infty}$ and $\nu_{\infty}$ are defined as in \eqref{muinftydef} and \eqref{nuinftydef} respectively, replacing $p$ by $q$.
Now, we fix a concentration point $x_{i}$, and for any $\rho>0$ we set $\psi_{\rho}(x)=\psi(\frac{x-x_{i}}{\rho})$, where $\psi\in C^{\infty}_{c}(\R^{N})$ is such that $0\leq \psi\leq 1$, $\psi=1$ in $B_{1}$, $\psi=0$ in $\R^{N}\setminus B_{2}$ and $|\nabla \psi|_{\infty}\leq 2$. Since $(u_{n}\psi_{\rho})$ is bounded in $X$, we get $\langle J'(u_{n}), u_{n}\psi_{\rho}\rangle=o_{n}(1)$ that is
\begin{align*}
&\iint_{\R^{2N}} \frac{|u_{n}(x)-u_{n}(y)|^{p-2}(u_{n}(x)-u_{n}(y))}{|x-y|^{N+sp}}(u_{n}(x)\psi_{\rho}(x)-u_{n}(y)\psi_{\rho}(y))dx dy+\int_{\R^{N}} |u_{n}|^{p}\psi_{\rho}dx \\
&+\iint_{\R^{2N}} \frac{|u_{n}(x)-u_{n}(y)|^{q-2}(u_{n}(x)-u_{n}(y))}{|x-y|^{N+sq}}(u_{n}(x)\psi_{\rho}(x)-u_{n}(y)\psi_{\rho}(y))dx dy+\int_{\R^{N}} |u_{n}|^{q}\psi_{\rho}dx \\
&=\lambda\int_{\R^{N}} h(x)f(u_{n}) u_{n}\psi_{\rho} dx+\int_{\R^{N}}|u_{n}|^{\q} \psi_{\rho} dx+o_{n}(1).
\end{align*}
Let us note that for $t\in \{p, q\}$ we have
\begin{align*}
&\iint_{\R^{2N}} \frac{|u_{n}(x)-u_{n}(y)|^{t-2}(u_{n}(x)-u_{n}(y))}{|x-y|^{N+st}}(u_{n}(x)\psi_{\rho}(x)-u_{n}(y)\psi_{\rho}(y))dx dy \\
&= \iint_{\R^{2N}} \frac{|u_{n}(x)-u_{n}(y)|^{t}}{|x-y|^{N+st}}\psi_{\rho}(x)dx dy+ \iint_{\R^{2N}} \frac{|u_{n}(x)-u_{n}(y)|^{t-2}(u_{n}(x)-u_{n}(y))}{|x-y|^{N+st}}u_{n}(y)(\psi_{\rho}(x)-\psi_{\rho}(y))dx dy
\end{align*}
so we can rewrite the above identity as follows
\begin{align}\label{10}
&\iint_{\R^{2N}} \frac{|u_{n}(x)-u_{n}(y)|^{p-2}(u_{n}(x)-u_{n}(y))}{|x-y|^{N+sp}}u_{n}(y)(\psi_{\rho}(x)-\psi_{\rho}(y))dx dy+\int_{\R^{N}} |u_{n}|^{p}\psi_{\rho}dx  \nonumber\\
&+\iint_{\R^{2N}} \frac{|u_{n}(x)-u_{n}(y)|^{q-2}(u_{n}(x)-u_{n}(y))}{|x-y|^{N+sq}}u_{n}(y)(\psi_{\rho}(x)-\psi_{\rho}(y))dx dy+\int_{\R^{N}} |u_{n}|^{q}\psi_{\rho}dx  \nonumber\\
&=- \iint_{\R^{2N}} \frac{|u_{n}(x)-u_{n}(y)|^{p}}{|x-y|^{N+sp}}\psi_{\rho}(x)dx dy-\iint_{\R^{2N}} \frac{|u_{n}(x)-u_{n}(y)|^{q}}{|x-y|^{N+sq}}\psi_{\rho}(x)dx dy   \nonumber\\
&+\lambda\int_{\R^{N}} h(x) f(u_{n}) u_{n}\psi_{\rho} dx+\int_{\R^{N}}|u_{n}|^{\q} \psi_{\rho} dx+o_{n}(1).
\end{align}
Now,
\begin{align*}
&\left|\iint_{\R^{2N}} \frac{|u_{n}(x)-u_{n}(y)|^{p-2}(u_{n}(x)-u_{n}(y))}{|x-y|^{N+sp}}u_{n}(y)(\psi_{\rho}(x)-\psi_{\rho}(y))dx dy\right| \\
&\leq [u_{n}]^{p-1}_{s,p} \left( \iint_{\R^{2N}} \frac{|\psi_{\rho}(x)-\psi_{\rho}(y)|^{p}}{|x-y|^{N+sp}} |u_{n}(y)|^{p} dx dy \right)^{1/p} \\
&\leq C\left( \iint_{\R^{2N}} \frac{|\psi_{\rho}(x)-\psi_{\rho}(y)|^{p}}{|x-y|^{N+sp}} |u_{n}(y)|^{p} dx dy \right)^{1/p},
\end{align*}
and using Lemma \ref{psilem} we obtain the following relations of limits
\begin{align}\label{11}
\lim_{\rho\rightarrow 0} \limsup_{n\rightarrow \infty} \iint_{\R^{2N}} \frac{|u_{n}(x)-u_{n}(y)|^{p-2}(u_{n}(x)-u_{n}(y))}{|x-y|^{N+sp}}u_{n}(y)(\psi_{\rho}(x)-\psi_{\rho}(y))dx dy=0
\end{align}
and
\begin{align}\label{12}
\lim_{\rho\rightarrow 0} \limsup_{n\rightarrow \infty} \iint_{\R^{2N}} \frac{|u_{n}(x)-u_{n}(y)|^{q-2}(u_{n}(x)-u_{n}(y))}{|x-y|^{N+sq}}u_{n}(y)(\psi_{\rho}(x)-\psi_{\rho}(y))dx dy=0.
\end{align}
On the other hand,
\begin{align}\label{13}
\lim_{\rho\rightarrow 0} \limsup_{n\rightarrow \infty} \int_{\R^{N}} |u_{n}|^{p} \psi_{\rho} dx=0=\lim_{\rho\rightarrow 0} \limsup_{n\rightarrow \infty} \int_{\R^{N}} |u_{n}|^{q} \psi_{\rho} dx,
\end{align}
and recalling that $h$ is bounded and $f$ has subcritical growth we get
\begin{align}\label{15}
\lim_{\rho\rightarrow 0} \limsup_{n\rightarrow \infty} \int_{\R^{N}} h(x) f(u_{n})u_{n} \psi_{\rho} dx=0.
\end{align}
Then, putting together \eqref{10}-\eqref{15} and using \eqref{FIG}, we can infer that 
\begin{equation}\label{nmi}
\mu_{i}\leq \nu_{i} \mbox{ for all } i\in I.
\end{equation}

Next, we verify that a similar inequality for $\mu_{\infty}$ and $\nu_{\infty}$ holds true. For this purpose, we use the function $\eta_{R}$ defined as in Lemma \ref{etaR}. Since $(u_{n}\eta_{R})$ is bounded in $X$, it follows that $\langle J'(u_{n}), u_{n}\eta_{R}\rangle=o_{n}(1)$, that is
\begin{align}\label{FS71}
&\iint_{\R^{2N}} \frac{|u_{n}(x)-u_{n}(y)|^{p}}{|x-y|^{N+sp}}\eta_{R}(x)dx dy+ \iint_{\R^{2N}} \frac{|u_{n}(x)-u_{n}(y)|^{p-2}(u_{n}(x)-u_{n}(y))}{|x-y|^{N+sp}}u_{n}(y)(\eta_{R}(x)-\eta_{R}(y))dx dy \nonumber\\
&+\iint_{\R^{2N}} \frac{|u_{n}(x)-u_{n}(y)|^{q}}{|x-y|^{N+sq}}\eta_{R}(x)dx dy+ \iint_{\R^{2N}} \frac{|u_{n}(x)-u_{n}(y)|^{q-2}(u_{n}(x)-u_{n}(y))}{|x-y|^{N+sq}}u_{n}(y)(\eta_{R}(x)-\eta_{R}(y))dx dy \nonumber\\
&+\int_{\R^{N}} |u_{n}|^{p}\eta_{R}dx+\int_{\R^{N}} |u_{n}|^{q}\eta_{R}dx =\lambda\int_{\R^{N}} h(x)f(u_{n}) u_{n}\eta_{R} dx+\int_{\R^{N}}|u_{n}|^{\q} \eta_{R} dx+o_{n}(1).
\end{align}
As before, using Lemma \ref{etaR} instead of Lemma \ref{psilem}, we have 
\begin{align}\label{FS72}
\lim_{R\rightarrow \infty} \limsup_{n\rightarrow \infty} \iint_{\R^{2N}} \frac{|u_{n}(x)-u_{n}(y)|^{p-2}(u_{n}(x)-u_{n}(y))}{|x-y|^{N+sp}}u_{n}(y)(\eta_{R}(x)-\eta_{R}(y))dx dy=0
\end{align}
and
\begin{align}\label{FS73}
\lim_{R\rightarrow \infty} \limsup_{n\rightarrow \infty} \iint_{\R^{2N}} \frac{|u_{n}(x)-u_{n}(y)|^{q-2}(u_{n}(x)-u_{n}(y))}{|x-y|^{N+sq}}u_{n}(y)(\eta_{R}(x)-\eta_{R}(y))dx dy=0.
\end{align}
On the other hand, in view of \eqref{f}, Theorem \ref{Sembedding}, H\"older inequality, the boundedness of $(u_{n})$ in $L^{p}(\R^{N})$ and $h\in L^{\infty}(\R^{N})\cap L^{\frac{\q}{\q-r}}(\R^{N})$, we can show that
\begin{equation}\label{FS74}
\lim_{R\rightarrow \infty}\limsup_{n\rightarrow \infty}\int_{\R^{N}} h(x) f(u_{n}) u_{n}\eta_{R} dx=0.
\end{equation}
Indeed, 
\begin{align*}
&\lim_{R\rightarrow \infty}\limsup_{n\rightarrow \infty}\left|\int_{\R^{N}} h(x) f(u_{n}) u_{n}\eta_{R} dx\right|\\
&\leq \lim_{R\rightarrow \infty}  \limsup_{n\rightarrow \infty} \left[|h|_{\infty}\e |u_{n}|^{p}_{p} +C_{\e} |h|_{\infty}|u_{n}|^{r}_{\q} \left(\int_{|x|>R} h(x)^\frac{\q}{\q-r} dx\right)^{\frac{\q-r}{\q}}\right] \\
&\leq \e C+C'_{\e} \lim_{R\rightarrow \infty} \left(\int_{|x|>R} h(x)^\frac{\q}{\q-r} dx\right)^{\frac{\q-r}{\q}}\\
&\leq \e C
\end{align*}
and from the arbitrariness of $\e>0$ we can see that \eqref{FS74} holds true.
Therefore, in view of \eqref{BS2} and \eqref{BS4} (with $p=q$), \eqref{FS71}-\eqref{FS74}, we obtain that 
\begin{equation}\label{nminfty}
\mu_{\infty}\leq \nu_{\infty}.
\end{equation}

Now we aim to show that $I=\emptyset$ and $\nu_{\infty}=\mu_{\infty}=0$. To achieve our goal, it is enough to prove that $\nu_{i}=0$ for all $i\in I$ and $\nu_{\infty}=0$. If by contradiction $\nu_{j}>0$ for some $j\in I$ or $\nu_{\infty}>0$, then we can use \eqref{9}, \eqref{93}, \eqref{nmi} and \eqref{nminfty} to deduce that $\nu_{j}\geq S_{*}^{\frac{N}{sq}}$ or $\nu_{\infty}\geq S_{*}^{\frac{N}{sq}}$. 
Hence, by $(f_3)$, $h\geq 0$ and \eqref{52FS} (with $p=q$) we get
\begin{align*}
c_{*}&=\lim_{n\rightarrow \infty}\left[J(u_{n})-\frac{1}{\theta}\langle J'(u_{n}), u_{n}\rangle\right] \\
&\geq   \lim_{n\rightarrow \infty}\left[\left(\frac{1}{\theta}-\frac{1}{\q}\right) |u_{n}|^{\q}_{\q}\right] \\
&\geq  \left(\frac{1}{\theta}-\frac{1}{\q}\right) \left[\sum_{i\in I} \nu_{i}+\nu_{\infty}\right],
\end{align*}
which yields
\begin{align*}
c_{*}\geq \left(\frac{1}{\theta}-\frac{1}{\q}\right) S_{*}^{\frac{N}{sq}}.
\end{align*}
This fact gives a contradiction in view of Lemma \ref{lem2}. 
Therefore \eqref{47FS} and \eqref{52FS} (with $p=q$), $\nu_{\infty}=\nu_{i}=0$ for all $i\in I$ yield $|u_{n}|_{\q}\rightarrow |u|_{\q}$ and using the Brezis-Lieb lemma \cite{BL} we have that $u_{n}\rightarrow u$ in $L^{\q}(\R^{N})$. 
By interpolation inequality and the boundedness of $(u_{n})$ in $X$ we have
\begin{equation}\label{qstar}
u_{n}\rightarrow u \mbox{ in } L^{t}(\R^{N}) \quad \forall t\in (p, \q].
\end{equation}
Moreover, it follows from the dominated convergence theorem that
\begin{equation}\label{DCT}
\int_{\R^{N}} |u_{n}|^{\q-2} u_{n} u \,dx\rightarrow \int_{\R^{N}} |u|^{\q} \,dx.
\end{equation}
Now, using \eqref{f}, $r\in (q, \q)$, the boundedness of $(u_{n})$ in $X$, $h\in L^{\infty}(\R^{N})$ and \eqref{qstar}, we can deduce that  
\begin{equation}\label{fun}
\lim_{n\rightarrow \infty} \int_{\R^{N}} h(x) f(u_{n}) (u_{n}-u) dx=0.
\end{equation}
At this point, we use the above relations of limits to show that $u_{n}\rightarrow u$ in $X$.
Taking into account $\langle J'(u_{n}), u_{n}\rangle=o_{n}(1)$, we have that
\begin{align}\label{APPA1}
\|u_{n}\|^{p}_{s,p}+\|u_{n}\|^{q}_{s,q}= \int_{\R^{N}} [\lambda h(x) f(u_{n}) u_{n} +|u_{n}|^{q^{*}_{s}}  ] dx+o_{n}(1).
\end{align}
On the other hand, $\langle J'(u_{n}), u \rangle=o_{n}(1)$ yields
\begin{align}\label{APPA2}
&\iint_{\R^{2N}} \frac{|u_{n}(x)-u_{n}(y)|^{p-2}(u_{n}(x)-u_{n}(y))}{|x-y|^{N+sp}}(u(x)-u(y))dx dy+\int_{\R^{N}} |u_{n}|^{p-2} u_{n}u dx  \nonumber\\
&+\iint_{\R^{2N}} \frac{|u_{n}(x)-u_{n}(y)|^{q-2}(u_{n}(x)-u_{n}(y))}{|x-y|^{N+sq}}(u(x)-u(y))dx dy+\int_{\R^{N}} |u_{n}|^{q-2}u_{n} u dx  \nonumber\\
&= \int_{\R^{N}} [\lambda h(x) f(u_{n}) u +|u_{n}|^{q^{*}_{s}-2}u_{n}u ] dx+o_{n}(1).
\end{align}
Combining \eqref{APPA1} with \eqref{APPA2} and using $u_{n}\rightharpoonup u$ in $X$, \eqref{DCT} and \eqref{fun}, we obtain
$$
\|u_{n}\|_{s,p}^{p}+\|u_{n}\|_{s,q}^{q}=\|u\|_{s,p}^{p}+\|u\|_{s,q}^{q}+o_{n}(1).
$$
Now, applying the Brezis-Lieb lemma \cite{BL} to the following sequences
$$
\frac{|u_{n}(x)-u_{n}(y)|}{|x-y|^{\frac{N+sp}{p}}}\in L^{p}(\R^{2N}), \quad \frac{|u_{n}(x)-u_{n}(y)|}{|x-y|^{\frac{N+sq}{q}}}\in L^{q}(\R^{2N}), \quad u_{n}\in L^{p}(\R^{N})\cap L^{q}(\R^{N}),
$$
we can see that
$$
\|u_{n}-u\|_{s, r}^{r}=\|u_{n}\|_{s,r}^{r}-\|u\|_{s,r}^{r}+o_{n}(1) \quad \mbox{ with } r\in \{p, q\},
$$
from which we can deduce that 
$$
\|u_{n}-u\|_{s, p}^{p}+\|u_{n}-u\|_{s,q}^{q}=o_{n}(1).
$$
Consequently, $u_{n}\rightarrow u$ in $X$ as $n\rightarrow \infty$.
\end{proof}

\noindent
Now, we are ready to give the proof of the main result of this work.
\begin{proof}[Proof of Theorem \ref{thm1}]
Applying the mountain pass theorem \cite{AR}, there exists $u\in X$ such that $J(u)=c_{*}$ and $J'(u)=0$ for all $\lambda$ sufficiently large. Since $c_{*}>0$, we infer that $u\not\equiv 0$. 
Now, we note that all the calculations done in the above lemmas can be repeated replacing $J$ by the functional
$$
J_{+}(u)=\frac{1}{p}\|u\|^{p}_{s,p}+\frac{1}{q}\|u\|^{q}_{s,q}-\lambda\int_{\R^{N}} h(x) F(u)dx-\frac{1}{q^{*}_{s}}|u^{+}|_{q^{*}_{s}}^{q^{*}_{s}}.
$$
Therefore, we can prove that \eqref{P} admits a nontrivial non-negative weak solution $u$. Indeed, using the facts $\langle J_{+}'(u), u^{-}\rangle=0$, $f(t)=0$ for $t\leq 0$ and $(h)$, where $u^{-}=\min\{u, 0\}$, we get
\begin{align*}
&\iint_{\R^{2N}} \frac{|u(x)-u(y)|^{p-2}(u(x)-u(y))}{|x-y|^{N+sp}} (u^{-}(x)-u^{-}(y)) dx dy+\int_{\R^{N}}|u^{-}|^{p} dx \nonumber\\
&+\iint_{\R^{2N}} \frac{|u(x)-u(y)|^{q-2}(u(x)-u(y))}{|x-y|^{N+sq}} (u^{-}(x)-u^{-}(y)) dx dy+\int_{\R^{N}}|u^{-}|^{q} dx \nonumber\\
&=0,
\end{align*}
which combined with the following elementary inequality
\begin{equation}\label{EQT}
|x-y|^{t-2}(x-y)(x^{-}-y^{-})\geq |x^{-}-y^{-}|^{t} \quad \forall x, y\in\R \quad \forall t> 1,
\end{equation}
yields $\|u^{-}\|_{s,p}^{p}+\|u^{-}\|_{s,q}^{q}\leq 0$. Therefore, $u^{-}=0$ in $\R^{N}$, that is $u\geq 0$ in $\R^{N}$. 
In conclusion, we have proved that \eqref{P} admits a nontrivial non-negative solution for all $\lambda$ sufficiently large.
\end{proof}

%\noindent {\bf Acknowledgements.} 
%The author is very grateful to the referees for their detailed comments which improved the presentation of the paper in an essential way.

\end{document}